\documentclass[11pt]{amsart}
\usepackage{amsmath,amsthm,amsfonts,amssymb,mathrsfs, mathtools}
\usepackage{enumerate}
\usepackage[colorlinks=true, linkcolor=blue, citecolor=magenta, menucolor=black]{hyperref}
\usepackage[demo]{graphicx}
\usepackage[up]{caption}
\usepackage{subcaption}
\usepackage{pgf,tikz}
\usepackage[toc,page]{appendix}
\usetikzlibrary{arrows}
\usetikzlibrary{patterns}
\usepackage{color}

\numberwithin{equation}{section}
\theoremstyle{plain}
\newtheorem{thm}{Theorem}[section]

\newtheorem{lem}[thm]{Lemma}
\newtheorem{prop}[thm]{Proposition}
\newtheorem{cor}[thm]{Corollary}

\newtheorem{letterthm}{Theorem}

\newtheorem{lettercor}[letterthm]{Corollary}

\theoremstyle{definition}

\newtheorem*{defn*}{Definition}

\newtheorem{rem}[thm]{Remark}

\newtheorem*{terminology}{Terminology}

\newtheorem*{shalom-conjecture}{Shalom's conjecture}
\newtheorem*{CRC}{Connes'\! rigidity conjecture}
\newcommand{\N}{\mathbb{N}}
\newcommand{\R}{\mathbb{R}}
\newcommand{\C}{\mathbb{C}}

\newcommand{\Q}{\mathbb{Q}}

\newcommand{\Map}{\operatorname{Map}}

\newcommand{\Tr}{\operatorname{Tr}}

\newcommand{\ovt}{\mathbin{\overline{\otimes}}}
\newcommand{\Aut}{\operatorname{Aut}}
\newcommand{\Ad}{\operatorname{Ad}}

\newcommand{\id}{\operatorname{id}}

\newcommand{\SL}{\operatorname{SL}}

\newcommand{\Prob}{\operatorname{Prob}}
\newcommand{\supp}{\operatorname{supp}}

\newcommand{\rk}{\operatorname{rk}}
\newcommand{\bary}{\operatorname{Bar}}
\newcommand{\dpr}{^{\prime\prime}}

\newcommand{\rB}{\operatorname{B}}

\newcommand{\rE}{\operatorname{ E}}
\newcommand{\rC}{\operatorname{C}}

\newcommand{\rL}{\operatorname{ L}}

\newcommand{\Out}{\operatorname{Out}}
\newcommand{\Inn}{\operatorname{Inn}}

\makeatletter
\@namedef{subjclassname@2020}{%
  \textup{2020} Mathematics Subject Classification}
\makeatother

\allowdisplaybreaks

\title[Weyl groups and rigidity of von Neumann algebras]{Weyl groups and rigidity of \\ von Neumann algebras}

\begin{document}

\begin{abstract}
Let $G$ be a noncompact semisimple algebraic group with trivial center, $S < G$ a maximal split torus, $H < G$ the centralizer of $S$ in $G$ and $\Gamma < G$ an irreducible lattice. Consider the group measure space von Neumann algebra $\mathscr M = \rL(\Gamma \curvearrowright G/H)$ associated with the nonsingular action $\Gamma \curvearrowright G/H$ and regard the group von Neumann algebra $M = \rL(\Gamma)$ as a von Neumann subalgebra $M \subset \mathscr M$. We show that the group $\Aut_M(\mathscr M)$ of all unital normal $\ast$-automorphisms of $\mathscr M$ acting identically on $M$ is isomorphic to the Weyl group $\mathscr W_G$ of the semisimple algebraic group $G$. Our main theorem is a noncommutative analogue of a rigidity result of Bader--Furman--Gorodnik--Weiss  for group actions on algebraic homogeneous spaces and moreover gives new insight towards Connes'\! rigidity conjecture for higher rank lattices.
\end{abstract}

\author{Cyril Houdayer}
\address{\'Ecole normale sup\'erieure \\ D\'epartement de math\'ematiques et applications \\ Universit\'e Paris-Saclay \\ 45 rue d'Ulm \\ 75230 Paris Cedex 05 \\ FRANCE}
\email{cyril.houdayer@ens.psl.eu}
\thanks{CH is supported by ERC Advanced Grant NET 101141693.}

\author{Adrian Ioana}
\address{Department of Mathematics \\ University of California at San Diego \\ 9500 Gilman Drive \\ La Jolla \\ CA
92093 \\ USA}
\email{aioana@ucsd.edu}
\thanks{AI is supported by NSF grants DMS-2153805 and DMS-2451697.}



\maketitle

\section{Introduction and statement of the main results}

In order to state our main results, we use the following terminology regarding algebraic groups. We refer to \cite{Bo91, Ma91} for further details.

\begin{terminology}
Let $k$ be a local field of characteristic zero, that is, $k$ is equal to $\R$, $\C$ or a finite extension of $\Q_p$ for some prime number $p$. Let $\mathbf G$ be a Zariski connected simply connected $k$-isotropic almost $k$-simple algebraic $k$-group. All maximal $k$-split tori of $\mathbf G$ are conjugate over $k$ (i.e.\! by elements of $\mathbf G(k)$). We choose $\mathbf S < \mathbf G$ a maximal $k$-split torus and $\mathbf P < \mathbf G$ a minimal parabolic $k$-subgroup such that $\mathbf H = \mathscr Z_{\mathbf G}(\mathbf S) < \mathbf P$. The dimension of $\mathbf S$ is called the $k$-rank of $\mathbf G$ and is denoted by $\rk_k(\mathbf G)$. Since $\mathbf G$ is $k$-isotropic, we have $\rk_k(\mathbf G) \geq 1$. The centralizing $k$-subgroup $\mathscr Z_{\mathbf G}(\mathbf S)$ is the Zariski connected component of the normalizing $k$-subgroup $\mathscr N_{\mathbf G}(\mathbf S)$. The finite group
$\mathscr W_{\mathbf G} = \mathscr N_{\mathbf G}(\mathbf S)/\mathscr Z_{\mathbf G}(\mathbf S)$ is called the {\em Weyl group} of $\mathbf G$ relative to $k$. Every coset of $\mathscr N_{\mathbf G}(\mathbf S)/\mathscr Z_{\mathbf G}(\mathbf S)$ is represented by an element rational over $k$, that is, $\mathscr N_{\mathbf G}(\mathbf S) = \mathscr N_{\mathbf G}(\mathbf S)(k) \cdot \mathscr Z_{\mathbf G}(\mathbf S)$. In the case when $\mathbf G = \SL_n$ for $n \geq 2$, we may assume that $\mathbf S < \mathbf G$ is the group of diagonal matrices, and we have $\mathbf H = \mathbf S$ and $\mathscr W_{\mathbf G} \cong \mathfrak S_n$.
\end{terminology}

We introduce the following notation which we will be using throughout our paper. Let $d \geq 1$. For every $i \in \{1, \dots, d\}$, let $k_i$ be a local field of characteristic zero, $\mathbf G_i$ a Zariski connected simply connected $k_i$-isotropic almost $k_i$-simple algebraic $k_i$-group as above and choose $\mathbf S_i < \mathbf P_i < \mathbf G_i$ accordingly so that $\mathbf H_i = \mathscr Z_{\mathbf G_i}(\mathbf S_i) < \mathbf P_i$. Set $G = \prod_{i = 1}^d \mathbf G_i(k_i)$, $S = \prod_{i = 1}^d \mathbf S_i(k_i)$, $H = \prod_{i = 1}^d \mathbf H_i(k_i)$, $P = \prod_{i = 1}^d \mathbf P_i(k_i)$ and $\mathscr W_G = \prod_{i = 1}^d \mathscr W_{\mathbf G_i}$. Endow the homogeneous space $G/H$ (resp.\! $G/P$) with its unique $G$-invariant measure class. Since $G$ and $H$ are both unimodular, the homogeneous space $G/H$ carries a $G$-invariant $\sigma$-finite infinite measure $m$ whose measure class coincides with the unique $G$-invariant measure class on $G/H$. It is well known that we have a natural isomorphism of nonsingular $G$-spaces $G/H \cong G/P \times G/P$, where the composition with the projection onto the first factor corresponds to the factor map $G/H \to G/P$ (see e.g.\! the discussion preceding \cite[Theorem 3]{BF11}). Let $\Gamma < G$ be an irreducible lattice and set $\Lambda = \Gamma/\mathscr Z(\Gamma)$. Since $\mathscr Z(\Gamma) < H \cap \mathscr Z(G)$, we have $\mathscr Z(\Gamma) < \ker(\Gamma \curvearrowright G/H)$ and so we may consider the well-defined nonsingular action $\Lambda \curvearrowright G/H$. Using the terminology on algebraic groups as above and since $H < G$ is noncompact and $\Gamma < G$ is irreducible, Moore's ergodicity theorem implies that $\Lambda \curvearrowright G/H$ is ergodic (see \cite{HM77}). Moreover, Lemma \ref{lem-free} below implies that $\Lambda \curvearrowright G/H$ is essentially free.

As a consequence of the main results of \cite{BFGW12} (see Theorem \ref{thm-equivariant} below for further details), any $\Lambda$-equivariant nonsingular automorphism of $G/H$ is automatically $G$-equivariant and so we have the following isomorphisms
$$\Aut_\Lambda(G/H) = \Aut_\Gamma(G/H) = \Aut_G(G/H) \cong \mathscr W_G.$$
The aim of our paper is to prove a noncommutative analogue of the above isomorphism result. Before stating our main theorem, we need to introduce some further notation.

We denote by $\sigma : \Lambda \curvearrowright \rL^\infty(G/H)$ the von Neumann algebraic action corresponding to the nonsingular action $\Lambda \curvearrowright G/H$. Denote by $M = \rL(\Lambda) = \left\{u_\gamma \mid \gamma \in \Lambda \right\}\dpr$ the group von Neumann algebra. Denote by $\mathscr M = \rL(\Lambda \curvearrowright G/H)$ the group measure space von Neumann algebra associated with the nonsingular action $\Lambda \curvearrowright G/H$ together with its Cartan subalgebra $\mathscr A = \rL^\infty(G/H)$ and regard $M \subset \mathscr M$ as a von Neumann subalgebra. For every $w = (w_i)_i \in \mathscr W_G$, choose $n = (n_i)_i \in \prod_{i = 1}^d \mathscr N_{\mathbf G_i}(\mathbf S_i)(k_i)$ so that $w = (n_i\mathbf H_i)_i \in \mathscr W_G$ and define $\theta_w \in \Aut(\mathscr A)$ as the unique unital normal $\ast$-automorphism that satisfies $\theta_w(F)(gH) = F(gHn^{-1})$ for every $F \in \mathscr A$ and almost every $gH \in G/H$. Then we have $\theta_w \circ \sigma_\gamma = \sigma_\gamma \circ \theta_w$ for every $\gamma \in \Lambda$ and every $w \in \mathscr W_G$.

Denote by $\Aut_M(\mathscr M)$ the group of all unital normal $\ast$-automorphisms $\Theta \in \Aut(\mathscr M)$ such that $\Theta|_M = \id_M$. For every $w \in \mathscr W_G$, define $\Theta_w \in \Aut_M(\mathscr M)$ as the unique unital normal $\ast$-automorphism that satisfies $\Theta_w|_M = \id_M$ and $\Theta_w|_{\mathscr A} = \theta_w$. Then $\rho : \mathscr W_G \to \Aut_M(\mathscr M) : w \mapsto \Theta_w$ is an injective group homomorphism.

Our main result is the following novel rigidity phenomenon for the inclusion of von Neumann algebras $M = \rL(\Lambda) \subset \rL(\Lambda \curvearrowright G/H) = \mathscr M$ associated with the nonsingular action $\Lambda \curvearrowright G/H$.

\begin{letterthm}\label{main-theorem-weyl}
The group homomorphism $\rho : \mathscr W_G \to \Aut_M(\mathscr M)$ is an isomorphism. In particular, we have 
$$\mathscr W_G \cong \Aut_M(\mathscr M).$$
\end{letterthm}

Since the inclusion $M \subset \mathscr M$ is irreducible (see Corollary \ref{cor-commutant} below), we have $\Inn_M(\mathscr M) = \left \{ \Ad(u) \mid u \in \mathscr U(M' \cap \mathscr M) \right \} = \{\id_{\mathscr M}\}$. This further implies that $\mathscr W_G \cong \Aut_M(\mathscr M) = \Out_M(\mathscr M)$. The striking feature of Theorem \ref{main-theorem-weyl} is that even though the inclusion $M = \rL(\Lambda) \subset \rL(\Lambda \curvearrowright G/H) = \mathscr M$ tends to forget the nonsingular action $\Lambda \curvearrowright G/H$, it still retains the Weyl group $\mathscr W_G$ via the group isomorphism $\mathscr W_G \cong \Aut_M(\mathscr M)$. 

This result is related to Boutonnet--Houdayer's noncommutative factor theorem \cite{BH22} (see also \cite{Ho21}). In that respect, set $\mathscr B = \rL(\Lambda \curvearrowright G/P)$ and naturally regard $M \subset \mathscr B \subset \mathscr M$. Using the same setting as above and under the higher rank assumption $\sum_{i = 1}^d \rk_{k_i}(\mathbf G_i) \geq 2$, it was shown in \cite[Theorem B]{BH22} that intermediate von Neumann subalgebras $M \subset \mathscr N \subset \mathscr B$ are in one-to-one correspondence with intermediate parabolic subgroups $P < Q < G$ via the identification $\mathscr N = \rL(\Lambda \curvearrowright G/Q)$. In particular, the inclusion $M \subset \mathscr B$ retains the abstract poset of all intermediate parabolic subgroups and the rank $\sum_{i = 1}^d \rk_{k_i}(\mathbf G_i)$. However, it is open whether the inclusion $M \subset \mathscr B$ retains the Weyl group $\mathscr W_G$.

Theorem \ref{main-theorem-weyl} above implies that the larger inclusion $M \subset \mathscr M$ retains the Weyl group $\mathscr W_G$ via the group isomorphism $\Aut_M(\mathscr M) \cong \mathscr W_G$ and therefore retains the abstract poset of intermediate parabolic subgroups $P < Q < G$. However, Theorem \ref{main-theorem-weyl} and \cite[Theorem B]{BH22} are of a different nature, their proofs do not rely on the same tools and neither theorem implies the other. We emphasize that there is no higher rank assumption in Theorem \ref{main-theorem-weyl} while the higher rank assumption is essential in the proof of \cite[Theorem B]{BH22} because one uses the noncommutative Nevo--Zimmer theorem \cite{BH19, BBH21} in the simple case ($d = 1$) and the dichotomy result for equivariant ucp maps \cite{BBHP20} in the product case ($d \geq 2$).

The present work is also motivated by Connes'\! rigidity conjecture for higher rank lattices in semisimple Lie groups.

\begin{CRC}
For every $j \in \{1, 2\}$, let $G_j$ be a semisimple connected real Lie group with trivial center and no compact factors such that $\rk_{\R}(G_j) \geq 2$ and let $\Gamma_j < G_j$ be an irreducible lattice. If $\rL(\Gamma_1) \cong \rL(\Gamma_2)$, then $G_1 \cong G_2$.
\end{CRC}

For every $j \in \{1, 2\}$, let $d_j \geq 1$ and for every $i \in \{1, \dots, d_j\}$, let $k_{i, j}$ be a local field of characteristic zero, $\mathbf G_{i, j}$ a Zariski connected simply connected $k_{i, j}$-isotropic almost $k_{i, j}$-simple algebraic $k_{i, j}$-group and $\mathbf S_{i, j} < \mathbf G_{i, j}$ a maximal $k_{i, j}$-split torus. Set $G_j = \prod_{i = 1}^{d_j} \mathbf G_{i, j}(k_{i, j})$, $H_j = \prod_{i = 1}^{d_j} \mathscr Z_{\mathbf G_{i, j}}(\mathbf S_{i, j})(k_{i, j})$ and $\mathscr W_{G_j} = \prod_{i = 1}^{d_j} \mathscr W_{\mathbf G_{i, j}}$. Let $\Gamma_j < G_j$ be an irreducible lattice and set $\Lambda_j = \Gamma_j/\mathscr Z(\Gamma_j)$. Following Theorem \ref{main-theorem-weyl}, denote by $\rho_j : \mathscr W_{G_j} \to \Aut_{\rL(\Lambda_j)}(\rL(\Lambda_j \curvearrowright G_j/H_j))$ the corresponding group isomorphism. We derive the following consequence of Theorem \ref{main-theorem-weyl} which gives new insight towards Connes'\! rigidity conjecture.

\begin{lettercor}\label{main-cor-connes}
Let $\Psi : \rL(\Lambda_1 \curvearrowright G_1/H_1) \to \rL(\Lambda_2 \curvearrowright G_2/H_2)$ be a surjective unital normal $\ast$-isomorphism such that $\Psi(\rL(\Lambda_1)) = \rL(\Lambda_2)$. Then the mapping $\mathscr W_{G_1} \to \mathscr W_{G_2} : w \mapsto \rho_2^{-1}(\Psi \circ \rho_1(w) \circ \Psi^{-1})$ is a group isomorphism.
\end{lettercor}

Let us say a few words about the proof of Theorem \ref{main-theorem-weyl}. Let $\Theta \in \Aut_M(\mathscr M)$. In order to prove that there exists $w \in \mathscr W_G$ such that $\Theta = \Theta_w$, it is enough to show that $\Theta(\mathscr A) = \mathscr A$, where $\mathscr A = \rL^\infty(G/H)$. To do this, we proceed in three steps. Firstly, we consider the equivariant normal ucp map $\Phi = \rE \circ \Theta|_{\mathscr A} : \mathscr A \to \mathscr A$, where $\rE : \mathscr M \to \mathscr A$ is the canonical faithful normal conditional expectation. By interpreting the main result of \cite{BFGW12} as a rigidity statement for equivariant normal ucp maps, we conclude that $\Phi = \sum_{w \in \mathscr W_G} \alpha_w \theta_w$ is a convex combination of
elements of the finite set $\left \{ \theta_w \mid w \in \mathscr W_G\right \}$. Secondly, we develop a convexity/maximality argument to show that there exist $w \in \mathscr W_G$ and $u \in \mathscr U(\mathscr M)$ such that $\Theta = \Ad(u) \circ \Theta_w$ on $\mathscr A$. This second step is purely noncommutative and is reminiscent of the intertwining method in Popa's deformation/rigidity theory \cite{Po06}. Thirdly, by exploiting the essential freeness of the nonsingular action $\Lambda \times \mathscr W_G \curvearrowright G/H$, we conclude that $u \in \mathscr U(\mathscr A)$, which finally implies that $\Theta = \Theta_w$.

\subsection*{Acknowledgments} CH is grateful to Amine Marrakchi for insightful conversations that led to Proposition \ref{prop-commutant} and to Fran\c{c}ois Thilmany for thought-provoking discussions on the structure and classification of simple algebraic groups. The authors also thank Uri Bader and Itamar Vigdorovich for their valuable comments.

\subsection*{Special acknowledgments} Part of this work was done while CH was facing a severe illness and undergoing aggressive treatment. He wishes to express his gratitude to his family, friends and colleagues for their constant support, and above all to his wife, for her love and strength through this challenging time. He is also deeply grateful to the medical staff at Institut Curie, who treated him with such care and attention. He owes them his full recovery.

\section{Preliminaries}

We keep the same notation as in the introduction. For any group action $L \curvearrowright Z$, we denote by $Z^L$ the set of all $L$-fixed points in $Z$. For any subgroup $L < R$, the quotient group $\mathscr N_{R}(L)/L$ is naturally identified with the group $\Aut_R(R/L)$ of all $R$-equivariant bijections of the set $R/L$ via the onto isomorphism
$$\mathscr N_{R}(L)/L \to \Aut_R(R/L) : n L \mapsto \left( gL \mapsto g Ln^{-1} \right).$$

\subsection{The Weyl group}

In this subsection, we fix a local field $k$ of characteristic zero, $\mathbf G$ a Zariski connected simply connected $k$-isotropic almost $k$-simple algebraic $k$-group and $\mathbf S < \mathbf G$ a maximal $k$-split torus. Set $\mathbf H = \mathscr Z_{\mathbf G}(\mathbf S)$. The Weyl group of $\mathbf G$ relative to $k$ is defined by the formula
$$\mathscr W_{\mathbf G} = \mathscr N_{\mathbf G}(\mathbf S)/\mathscr Z_{\mathbf G}(\mathbf S).$$ 

\begin{prop}\label{prop-weyl}
Keep the same notation as above. The following assertions hold:
\begin{itemize}
\item [$(\rm i)$] We have $\mathscr N_{\mathbf G}(\mathbf S)(k) = \mathscr N_{\mathbf G}(\mathbf H)(k)$.
\item [$(\rm ii)$] We have the natural identifications 
$$\mathscr W_{\mathbf G} \cong \mathscr N_{\mathbf G(k)}(\mathbf S(k))/\mathbf H(k) = \mathscr N_{\mathbf G(k)}(\mathbf H(k))/\mathbf H(k).$$
In particular, $\mathscr W_{\mathbf G} \cong \Aut_{\mathbf G(k)}(\mathbf G(k)/\mathbf H(k))$ coincides with the group of all nonsingular $\mathbf G(k)$-equivariant automorphisms of $\mathbf G(k)/\mathbf H(k)$.
\item [$(\rm iii)$] We have $(\mathbf G(k)/\mathbf H(k))^{\mathbf H(k)} = \mathscr N_{\mathbf G(k)}(\mathbf H(k))/\mathbf H(k) \cong \mathscr W_{\mathbf G}$.
\end{itemize}
\end{prop}

\begin{proof} By \cite[Proposition 1.4]{Sh97}, every algebraic $k$-group $\mathbf L$ admits a maximal (with respect to inclusion) $k$-subgroup $\mathbf N$, called the {\it $k$-discompact radical} of $\mathbf L$, which has no nontrivial $k$-compact quotients. Moreover, $\mathbf N<\mathbf L$ is a characteristic (hence, normal) subgroup and the quotient $\mathbf L/\mathbf N$ is $k$-compact.

We claim that $\mathbf S$ is the $k$-discompact radical of $\mathbf H$. Indeed, denote by $\mathbf R$ the $k$-discompact radical of $\mathbf H$. Since $\mathbf S$ has no nontrivial $k$-compact quotient, it follows that $\mathbf S < \mathbf R$. Since $\mathbf S$ is central in $\mathbf H$, we have that $\mathbf R/\mathbf S$ is a $k$-subgroup of $\mathbf H/\mathbf S$. Since the characteristic of $k$ is zero, the quotient group $\mathbf H(k)/\mathbf S(k)$ has finite index in $(\mathbf H/\mathbf S)(k)$. Since $\mathbf H(k)/\mathbf S(k)$ is compact by \cite[Proposition I.2.3.6]{Ma91}, it follows that $(\mathbf H/\mathbf S)(k)$ is compact. Then $(\mathbf R/\mathbf S)(k)$ is compact and since $\mathbf R$ is the $k$-discompact radical of $\mathbf H$, we necessarily have $\mathbf S = \mathbf R$.

$(\rm i)$ It is plain to see that $\mathscr N_{\mathbf G}(\mathbf S)(k) < \mathscr N_{\mathbf G}(\mathbf H)(k)$. Conversely, let $n \in \mathscr N_{\mathbf G}(\mathbf H)(k)$. Then $\iota_n : \mathbf H \to \mathbf H : h \mapsto n h n^{-1}$ is a $k$-automorphism. Since the $k$-discompact radical $\mathbf S$ is a characteristic subgroup of $\mathbf H$, it follows that $n \mathbf S n^{-1} = \mathbf S$ and so $n \in \mathscr N_{\mathbf G}(\mathbf S)(k)$.

$(\rm ii)$ Since $\mathbf S$ is a Zariski connected $k$-group, $\mathbf S(k)$ is Zariski dense in $\mathbf S$ by \cite[Proposition I.2.5.3(ii)]{Ma91}. This implies that $\mathscr N_{\mathbf G(k)}(\mathbf S(k)) =  \mathscr N_{\mathbf G}(\mathbf S)(k)$. Likewise, since $\mathbf H$ is a Zariski connected $k$-group, we have $\mathscr N_{\mathbf G(k)}(\mathbf H(k)) =  \mathscr N_{\mathbf G}(\mathbf H)(k)$. Using $(\rm i)$, we obtain $\mathscr N_{\mathbf G(k)}(\mathbf S(k)) = \mathscr N_{\mathbf G(k)}(\mathbf H(k))$ and so 
$$\mathscr N_{\mathbf G(k)}(\mathbf S(k))/\mathbf H(k) = \mathscr N_{\mathbf G(k)}(\mathbf H(k))/\mathbf H(k).$$
We know that the image of $\mathscr N_{\mathbf G}(\mathbf S)(k)$ under the epimorphism $\mathscr N_{\mathbf G}(\mathbf S) \to \mathscr W_{\mathbf G}$ is all of $\mathscr W_{\mathbf G}$. This further implies that $\mathscr N_{\mathbf G}(\mathbf S)(k)/\mathbf H(k)$ coincides with the $\mathscr N_{\mathbf G}(\mathbf S)(k)$-orbit of $\mathbf H$ in $\mathscr N_{\mathbf G}(\mathbf S)/\mathbf H$. Therefore, we have the natural identifications 
$$\mathscr W_{\mathbf G} \cong \mathscr N_{\mathbf G(k)}(\mathbf S(k))/\mathbf H(k) = \mathscr N_{\mathbf G(k)}(\mathbf H(k))/\mathbf H(k).$$
The moreover part follows from the observation made at the beginning of the section.

$(\rm iii)$ It is plain to see that $\mathscr N_{\mathbf G(k)}(\mathbf H(k))/\mathbf H(k) \subset (\mathbf G(k)/\mathbf H(k))^{\mathbf H(k)}$. Conversely, let $g \mathbf H(k) \in (\mathbf G(k)/\mathbf H(k))^{\mathbf H(k)}$. Then we have $g^{-1} \mathbf H(k) g < \mathbf H(k)$. Since $\mathbf H$ is a Zariski connected $k$-group, $\mathbf H(k)$ is Zariski dense in $\mathbf H$ by \cite[Proposition I.2.5.3(ii)]{Ma91} and so $g^{-1} \mathbf H g < \mathbf H$. This implies that the sequence $(g^{-n} \mathbf H g^n)_{n \geq 1}$ is a descending chain of algebraic subgroups of $\mathbf G$. By the descending chain condition, this sequence has finite length and so there exists $n \geq 1$ such that $g^{-n} \mathbf H g^n = g^{-(n+1)} \mathbf H g^{(n+ 1)}$. This further implies that $g^{-1} \mathbf H g = \mathbf H$ and so $g^{-1} \mathbf H(k) g = \mathbf H(k)$. Thus, we have $g \mathbf H(k) \in \mathscr N_{\mathbf G(k)}(\mathbf H(k))/\mathbf H(k)$.
\end{proof}

\subsection{Algebraic groups and homogeneous spaces}

We now come back to the notation we introduced in the introduction. Firstly, we review the work of \cite{BFGW12} regarding the rigidity of group actions on algebraic homogeneous spaces. For every $i \in \{1, \dots, d\}$, let $\mathbf L_i < \mathbf G_i$ be a $k_i$-subgroup and $\mathbf V_i$ a $k_i$-$\mathbf G_i$-algebraic variety. Set $L =  \prod_{i = 1}^d \mathbf L_i(k_i)$ and $V = \prod_{i = 1}^d \mathbf V_i(k_i)$. Define $\Prob^0(V) = V$ and $\Prob^{n + 1}(V) = \Prob(\Prob^n(V))$ for every $n \in \N$. Then for every $n \in \N$, we may regard $\Prob^n(V)$ as a Borel $\Gamma$-space (resp.\! $G$-space) and we denote by $\Map_\Gamma(G/L, \Prob^n(V))$ (resp.\! $\Map_G(G/L, \Prob^n(V))$) the space of all equivalence classes of $\Gamma$-equivariant (resp.\! $G$-equivariant) Borel maps $\beta : G/L \to \Prob^n(V)$. Assume moreover that $L < G$ is noncompact. Since $\Gamma < G$ is irreducible, Moore's ergodicity theorem implies that $\Gamma \curvearrowright G/L$ is ergodic (see \cite{HM77}). 

\begin{thm}[{\cite[Theorem 1.6]{BFGW12}}]\label{thm-equivariant}
Keep the same notation as above. Then for every $n \in \N$, we have
$$\Map_\Gamma(G/L, \Prob^n(V)) = \Map_G(G/L, \Prob^n(V)) \cong \Prob^n(V^L),$$
where $w \in \Prob^n(V^L)$ corresponds to the measurable $G$-equivariant map $\Phi_w : G/L \to \Prob^n(V) : gL \mapsto g \cdot w$.
 
In particular, in the case when $L = H$, we have 
$$\Aut_\Lambda(G/H) = \Aut_\Gamma(G/H) = \Aut_G(G/H) \cong \mathscr W_G.$$
\end{thm}

\begin{proof}
Let $\Phi : G/L \to \Prob^n(V)$ be a measurable $\Gamma$-equivariant map. We need to prove that there exists $w \in \Prob^n(V^L)$ such that $\Phi = \Phi_w$ almost everywhere, where $\Phi_w : G/L \to \Prob^n(V) : gL \mapsto g \cdot w$ is the measurable $G$-equivariant map associated with $w \in \Prob^n(V^L)$. 

By assumption, for every $i \in \{1, \dots, d\}$, $\mathbf G_i$ is a Zariski connected simply connected $k_i$-isotropic almost $k_i$-simple algebraic $k_i$-group. Then \cite[Theorem I.1.5.6 and Corollary I.5.6.7]{Ma91} imply that any proper normal subgroup of $\mathbf G_i(k_i)$ is contained in $\mathscr Z(\mathbf G_i)$ and $\mathbf G_i(k_i)$ has no proper finite index subgroup. This further implies that $G$ has no nontrivial algebraic compact factor group ($G$ satisfies condition $(\ast)$ in \cite{BFGW12}) and $G$ has no proper finite index subgroup. By applying \cite[Theorem 1.6]{BFGW12} (see also \cite[Corollary 7.1]{BFGW12}) and using the notation therein, we have that $M = G$ which implies that $M \cap L = L$ and so there exists $w \in \Prob^n(V^L)$ such that $\Phi = \Phi_w$ almost everywhere.

Assume now that $L = H$. Using Proposition \ref{prop-weyl}$(\rm ii), (\rm iii)$, we have the natural identifications $(G/H)^H =  \mathscr N_{G}(H)/H = \mathscr W_G$. We may now apply the previous paragraph to obtain
$$\Aut_\Lambda(G/H) = \Aut_\Gamma(G/H) = \Aut_G(G/H) \cong \mathscr W_G.$$
This finishes the proof.
\end{proof}

Since the Weyl group $\mathscr W_G$ coincides with the group $\Aut_\Lambda(G/H)$ of all $\Lambda$-equivariant nonsingular automorphisms of $G/H$, we may consider the well-defined nonsingular action $\Lambda \times \mathscr W_G \curvearrowright G/H$. The next lemma will be crucial in the proof of our main result.

\begin{lem}\label{lem-free}
The nonsingular action $\Lambda \times \mathscr W_G \curvearrowright G/H$ is essentially free.
\end{lem}

\begin{proof}
Let $(\lambda, w) \in \Lambda \times \mathscr W_G \setminus \{(e, e)\}$. We need to prove that the fixed-point subset $$\left \{gH \in G/H \mid (\lambda, w) \cdot gH = gH \right \}$$ is null in $G/H$. Choose $\gamma = (\gamma_i)_i \in G = \prod_{i = 1}^d \mathbf G_i(k_i)$ so that $\lambda = \gamma \mathscr Z(\Gamma)$ and $n = (n_i)_i \in \prod_{i = 1}^d \mathscr N_{\mathbf G_i}(\mathbf S_i)(k_i)$ so that $w = (n_i \mathbf H_i)_i$. Since $(\lambda, w) \neq (e, e)$, there exists $i \in \{1, \dots, d\}$ such that $\gamma_i \notin \mathscr Z(\mathbf G_i)$ or $n_i \notin \mathbf H_i$. Denote by $\pi_i : \mathbf G_i \to \mathbf G_i/\mathbf H_i$ the canonical algebraic $k_i$-morphism. Set $G_i = \mathbf G_i(k_i)$, $H_i = \mathbf H_i(k_i)$ and regard $G_i/H_i = \mathbf G_i(k_i)/\mathbf H_i(k_i) = \mathbf G_i(k_i) \pi_i(e) \subset (\mathbf G_i/\mathbf H_i)(k_i)$ as a closed and open subset (see e.g.\! \cite[Proposition I.2.1.4]{Ma91}). Denote by $p_i = {\pi_i}|_{G_i} : G_i \to G_i/H_i$ the corresponding restriction map. It suffices to show that the fixed-point subset 
$$W_i = \left \{gH_i \in G_i/H_i \mid \gamma_i gH_i n_i = gH_i \right \}$$ 
is null in $G_i/H_i$. By the general theory of locally compact second countable groups, $W_i$ is null in $G_i/H_i$ if and only if $V_i = p_i^{-1}(W_i)$ is null in $G_i$.

Denote by $\mathbf W_i$ the Zariski closure of $W_i$ in $\mathbf G_i/\mathbf H_i$ and observe that $\mathbf W_i \subset \left\{g\mathbf H_i \in \mathbf G_i/\mathbf H_i \mid \gamma_i g\mathbf H_i n_i = g\mathbf H_i\right\}$. By \cite[Theorem 14.4]{Bo91}, $\mathbf W_i \subset \mathbf G_i/\mathbf H_i$ is an algebraic subvariety defined over $k_i$ for which $W_i \subset G_i/H_i \cap \mathbf W_i(k_i)$. We claim that $\mathbf W_i$ is a proper subvariety of $\mathbf G_i/\mathbf H_i$. Indeed, otherwise if $\mathbf W_i = \mathbf G_i/\mathbf H_i$, then we have 
$$\mathbf G_i/\mathbf H_i = \left \{ g\mathbf H_i \in \mathbf G_i/\mathbf H_i \mid \gamma_i g\mathbf H_i n_i = g\mathbf H_i \right \}.$$
In particular, we have $n_i \mathbf H_i = \mathbf H_i n_i = \gamma_i^{-1} \mathbf H_i = \mathbf H_i \gamma_i^{-1}$ and so $\gamma_i g  \mathbf H_i \gamma_i^{-1} = g\mathbf H_i$ for every $g\mathbf H_i \in \mathbf G_i/\mathbf H_i$. Since $\mathbf H_i < \mathbf P_i$, this further implies that $\gamma_i \mathbf P_i \gamma_i^{-1} = \mathbf P_i$. Since $\mathscr N_{\mathbf G_i}(\mathbf P_i) = \mathbf P_i$, it follows that $\gamma_i \in \mathbf P_i$. Then we have $n_i \in \mathscr N_{\mathbf G_i}(\mathbf S_i) \cap \mathbf P_i = \mathscr Z_{\mathbf G_i}(\mathbf S_i) = \mathbf H_i$. Moreover, for every $g\mathbf H_i \in \mathbf G_i/\mathbf H_i$, we have $\gamma_i g \mathbf H_i = g \mathbf H_i$ and so $\gamma_i \in \bigcap_{g \in \mathbf G_i} g \mathbf H_i g^{-1}$. Since $\mathbf H_i < \mathbf G_i$ is a proper $k_i$-subgroup, $\bigcap_{g \in \mathbf G_i} g \mathbf H_i g^{-1} \lhd \mathbf G_i$ is a proper normal $k_i$-closed subgroup. Since $\mathbf G_i$ is almost $k_i$-simple, we have $\bigcap_{g \in \mathbf G_i} g \mathbf H_i g^{-1} = \mathscr Z(\mathbf G_i)$ and so $\gamma_i \in \mathscr Z(\mathbf G_i)$. Therefore, we showed that $\gamma_i \in \mathscr Z(\mathbf G_i)$ and $n_i \in \mathbf H_i$. This is a contradiction and so $\mathbf W_i$ is a proper algebraic subvariety of $\mathbf G_i/\mathbf H_i$.

It follows that $\mathbf U_i = \pi_i^{-1}(\mathbf W_i) $ is a proper algebraic subvariety of $\mathbf G_i$ such that $V_i = p_i^{-1}(W_i) \subset \pi_i^{-1}(\mathbf W_i) \cap G_i = \mathbf U_i \cap G_i$. Denote by $\mathbf V_i$ the Zariski closure of $\mathbf U_i \cap G_i$ in $\mathbf G_i$. By \cite[Theorem 14.4]{Bo91}, $\mathbf V_i \subset \mathbf G_i$ is an algebraic subvariety defined over $k_i$ such that $\mathbf V_i \subset \mathbf U_i$ and $V_i \subset \mathbf U_i \cap G_i \subset \mathbf V_i(k_i)$. Since $\mathbf G_i$ is connected, \cite[Proposition I.2.5.3$(\rm ii)$]{Ma91} implies that $\mathbf V_i(k_i)$ is null in $\mathbf G_i(k_i)$. This further implies that $V_i$ is null in $G_i$ and so $W_i$ is null in $G_i/H_i$.
\end{proof}

\subsection{Normal ucp maps and boundary maps}

We recall the following well-known result on equivariant normal ucp maps between abelian von Neumann algebras.

\begin{prop}\label{prop-ucp}
Let $L$ be a locally compact second countable group. Let $(X, \nu_X)$ be a standard probability $L$-space and $(Y, \nu_Y)$ a locally compact second countable Hausdorff topological $L$-space endowed with a fully supported Borel probability measure. 

To any $L$-equivariant normal ucp map $\Phi : \rL^\infty(Y, \nu_Y) \to \rL^\infty(X, \nu_X)$ corresponds an essentially unique $L$-equivariant measurable map $\beta : X \to \Prob(Y) : x \mapsto \beta_x$ such that $\bary(\beta_\ast \nu_X) = \nu_X \circ \Phi \prec \nu_Y$ and for $\nu_X$-almost every $x \in X$ and every $f \in \rC_0(Y)$, we have 
$$\beta_x(f) = \Phi(f)(x).$$
\end{prop}

\begin{proof} Let $\Phi : \rL^\infty(Y, \nu_Y) \to \rL^\infty(X, \nu_X)$ be a $L$-equivariant normal ucp map.

Firstly, we assume that $Y$ is compact. Since $\nu_Y$ is fully supported on $Y$, we may regard $\rC(Y) \subset \rL^\infty(Y, \nu_Y)$ as a unital $\rC^*$-subalgebra. Since $Y$ is a compact second countable Hausdorff topological space, $Y$ is metrizable and thus $\rC(Y)$ is $\|\cdot\|_\infty$-separable. Denote by $D \subset \rL^\infty(X, \nu_X)$ the $L$-invariant $\|\cdot\|_\infty$-separable unital $\rC^*$-subalgebra of $\rL^\infty(X, \nu_X)$ generated by the subspace $\Phi(\rC(Y))$. Observe that the action $L \curvearrowright D$ is $\|\cdot\|_\infty$-continuous. Denote by $Z$ the spectrum of $D$ and by $\nu_Z \in \Prob(Z)$ the Borel probability measure corresponding to the state $\nu_X|_{D} \in \mathfrak S(D)$. Then $Z$ is a compact metrizable topological $L$-space and we have  a $L$-equivariant measurable factor map $\pi : (X, \nu_X) \to (Z, \nu_Z)$. Then to the $L$-equivariant ucp map $\Psi : \rC(Y) \to \rC(Z) : f \mapsto \Phi(f)$ corresponds a unique $L$-equivariant continuous map $\alpha : Z \to \Prob(Y) : z \mapsto \alpha_z$ such that $\bary(\alpha_\ast \nu_Z) = \nu_Z \circ \Psi \prec \nu_Y$ and for every $z \in Z$ and every $f \in \rC(Y)$, we have $\alpha_z(f) = \Psi(f)(z)$. Then the $L$-equivariant measurable map $\beta : X \to \Prob(Y) : x \mapsto \alpha_{\pi(x)}$ does the job.

Secondly, we assume that $Y$ is noncompact. Since $\nu_Y$ is fully supported on $Y$, we may regard $\rC_0(Y) \subset \rL^\infty(Y, \nu_Y)$ as a $\rC^*$-subalgebra. Denote by $Z = Y \sqcup \{\infty\}$ the one-point compactification of $Y$, which is a compact metrizable space. Define the Borel probability measure $\nu_Z \in \Prob(Z)$ by $\nu_Z|_Y = \nu_Y$ and $\nu_Z(\{\infty\}) = 0$. Regard $\Prob(Y) \subset \Prob(Z)$ as a Borel subset. Then we have $\rL^\infty(Y, \nu_Y) = \rL^\infty(Z, \nu_Z)$. Let $\Phi : \rL^\infty(Y, \nu_Y) \to \rL^\infty(X, \nu_X)$ be an $L$-equivariant normal ucp map. Since the result holds for the compact metrizable space $Z$, there exists an essentially unique $L$-equivariant measurable map $\beta : X \to \Prob(Z) : x \mapsto \beta_x$ such that $\bary(\beta_\ast \nu_X) = \nu_X \circ \Phi \prec \nu_Z$ and for $\nu_X$-almost every $x \in X$ and every $f \in \rC(Z)$, we have $\beta_x(f) = \Phi(f)(x)$. It suffices to prove that $\beta_x(\{\infty\}) = 0$ for $\nu_X$-almost every $x \in X$.

We may find a nondecreasing sequence $f_n \in \rC_0(Y)$ of compactly supported functions such that $0 \leq f_n \leq 1$ for every $n \in \N$ and $f_n \to \mathbf 1_Y$ pointwise. For every $n \in \N$, denote by $K_n = \supp(f_n)$ the compact support in $Y$ of $f_n \in \rC_0(Y)$. Lebesgue's dominated convergence theorem implies that $f_n \to \mathbf 1_Y$ strongly in $\rL^\infty(Y, \nu_Y)$. Since $\Phi : \rL^\infty(Y, \nu_Y) \to \rL^\infty(X, \nu_X)$ is normal, we have $\Phi(f_n) \to \mathbf 1_X$ strongly in $\rL^\infty(X, \nu_X)$. Upon taking a subsequence, we may further assume that there is a conull Borel subset $X_0 \subset X$ for which $\beta_x(f_n) = \Phi(f_n)(x) \to 1$ for every $x \in X_0$. Then for every $x \in X_0$, we have $\beta_x(f_n) \leq \beta_x(K_n) \leq 1$ for every $n \in \N$ and so $\lim_n \beta_x(K_n) = 1$. Since $(K_n)_n$ is nondecreasing, letting $K = \bigcup_{n \in \N} K_n \subset Y$, for every $x \in X_0$, we have that $\beta_x(K) = 1$. Thus, for every $x \in X_0$, we obtain $\beta_x(\{\infty\}) \leq \beta_x(Z \setminus K) = \beta_x(Z) - \beta_x(K) = 0$.
\end{proof}

As a corollary to Theorem \ref{thm-equivariant} and Proposition \ref{prop-ucp}, we obtain the following rigidity result for equivariant normal ucp maps.

\begin{thm}\label{thm-rigidity-ucp}
Let $\Phi : \rL^\infty(G/H) \to \rL^\infty(G/H)$ be a $\Lambda$-equivariant normal ucp map. Then there exists a unique tuple $(\alpha_w)_{w \in \mathscr W_G} \in [0, 1]^{\mathscr W_G}$ such that $\sum_{w \in \mathscr W_G} \alpha_w = 1$ and 
$$\Phi = \sum_{w \in \mathscr W_G} \alpha_w \theta_w.$$
\end{thm}

\begin{proof}
Fix a Borel probability measure $\nu \in \Prob(G/H)$ whose measure class coincides with the unique $G$-invariant measure class on $G/H$. Using Proposition \ref{prop-ucp}, there exists an essentially unique $\Lambda$-equivariant measurable map $\beta : G/H \to \Prob(G/H) : b \mapsto \beta_b$ such that $\bary(\beta_\ast \nu) = \nu \circ \Phi \prec \nu$ and for $\nu$-almost every $b \in G/H$ and every $f \in \rC_0(G/H)$, we have $\beta_b(f) = \Phi(f)(b)$. By Proposition \ref{prop-weyl}$(\rm ii), (\rm iii)$, we have the natural identifications $(G/H)^H =  \mathscr N_{G}(H)/H = \mathscr W_G$. By Theorem \ref{thm-equivariant} and since $\Prob((G/H)^H)$ is a finite dimensional simplex, there exists a unique tuple $(\alpha_w)_{w \in \mathscr W_G} \in [0, 1]^{\mathscr W_G}$ such that $\sum_{w \in \mathscr W_G} \alpha_w = 1$ and $\beta_b = \sum_{w \in \mathscr W_G} \alpha_w \delta_{w^{-1} b}$ for $\nu$-almost every $b \in G/H$. This means exactly that $\Phi = \sum_{w \in \mathscr W_G} \alpha_w \theta_w$.
\end{proof}

\subsection{The group measure space construction}

Let $L$ be a locally compact second countable group, $(X, \nu)$ a standard probability space and $L \curvearrowright (X, \nu)$ a nonsingular action. Following \cite[Definition 6.5]{BG14}, we say that $L \curvearrowright (X, \nu)$ is {\em metrically ergodic} if whenever $L \curvearrowright (Z, d)$ is a continuous isometric action on a separable metric space, every $L$-equivariant measurable map $F : X \to Z$ is $\nu$-almost everywhere constant. 

We prove the following general fact about the group measure space construction $\rL(\Upsilon \curvearrowright X)$ associated with a metrically ergodic nonsingular action $\Upsilon \curvearrowright (X, \nu)$.

\begin{prop}\label{prop-commutant}
Let $\Upsilon$ be a countable discrete group, $(X, \nu)$ a standard probability space and $\Upsilon \curvearrowright (X, \nu)$ a metrically ergodic nonsingular action. Then we have $\rL(\Upsilon)^\prime \cap \rL(\Upsilon \curvearrowright X) = \mathscr Z(\rL(\Upsilon))$.
\end{prop}

\begin{proof}
Denote by $\kappa : \Upsilon \to \mathscr U(\rL^2(X, \nu))$ the Koopman representation of the nonsingular action $\Upsilon \curvearrowright (X, \nu)$. Denote by $\sigma : \Upsilon \curvearrowright \rL^\infty(X, \nu)$ the von Neumann algebraic action corresponding to the nonsingular action $\Upsilon \curvearrowright (X, \nu)$. Then we have $\sigma_\gamma(F) = \kappa_\gamma F \kappa_\gamma^*$ for every $F \in \rL^\infty(X, \nu)$ and every $\gamma \in \Upsilon$. Set $\mathscr H = \rL^2(X, \nu) \otimes \ell^2(\Upsilon)$. Define $\pi : \rL^\infty(X) \to \rB(\mathscr H) : F \mapsto \sum_{\gamma \in \Upsilon} \sigma_\gamma(F) \otimes p_\gamma$, where $p_\gamma : \ell^2(\Upsilon) \to \C \delta_\gamma$ is the rank one projection corresponding to $\gamma \in \Upsilon$. We may regard the group measure space construction $\rL(\Upsilon \curvearrowright X)$ as the von Neumann subalgebra of $\rB(\mathscr H)$ generated by $\pi(F)$ and $1 \otimes \lambda_\gamma$ for $F \in \rL^\infty(X)$ and $\gamma \in \Upsilon$. Then it is well known that the commutant $\rL(\Upsilon \curvearrowright X)^\prime$ is generated by $F \otimes 1$ and $\kappa_\gamma \otimes J\lambda_\gamma J$ for $F \in \rL^\infty(X, \nu)$ and $\gamma \in \Upsilon$. Since $\rL(\Upsilon \curvearrowright X) = \rL(\Upsilon \curvearrowright X)\dpr$, we may view $\rL(\Upsilon \curvearrowright X) = (\rL^\infty(X) \ovt \rB(\ell^2(\Upsilon)))^\Upsilon$ as the fixed-point von Neumann subalgebra, where the action $\Upsilon \curvearrowright \rL^\infty(X) \ovt \rB(\ell^2(\Upsilon))$ is defined by the formula
$$\forall \gamma \in \Upsilon, \forall F \in \rL^\infty(X) \ovt \rB(\ell^2(\Upsilon)), \quad \gamma \cdot F = \left(\sigma_\gamma \otimes \Ad(J \lambda_\gamma J) \right)(F).$$
Then we may regard $\rL(\Upsilon)^\prime \cap \rL(\Upsilon \curvearrowright X)$ as the von Neumann algebra of all equivalence classes of $\Upsilon$-equivariant measurable functions $F : X \to J \rL(\Upsilon) J$, where we consider the action $\rho : \Upsilon \curvearrowright J \rL(\Upsilon) J$ given by $\rho_\gamma=\Ad(J\lambda_\gamma J)$ for $\gamma \in \Upsilon$. If we endow $J \rL(\Upsilon) J$ with the $\|\cdot\|_2$-norm, then $J \rL(\Upsilon) J$ is separable and $\rho$ is an isometric action. 

Since $\Upsilon \curvearrowright (X, \nu)$ is metrically ergodic, any $\Upsilon$-equivariant measurable function $F : X \to J \rL(\Upsilon) J$ is $\nu$-almost everywhere constant. This shows that $\rL(\Upsilon)^\prime \cap \rL(\Upsilon \curvearrowright X) = \mathscr Z(\rL(\Upsilon))$.
\end{proof}

We keep the same notation as in the introduction with $\Lambda = \Gamma/\mathscr Z(\Gamma)$, $M = \rL(\Lambda)$, $\mathscr A = \rL^\infty(G/H)$ and $\mathscr M = \rL(\Lambda \curvearrowright G/H)$. We recall the following well-known fact.

\begin{lem}\label{lem-icc}
The group $\Lambda$ has infinite conjugacy classes.
\end{lem}

\begin{proof}
Let $\gamma \in \Gamma$ be such that its conjugacy class $C(\gamma) = \left \{ h \gamma h^{-1} \mid h \in \Gamma\right \}$ is finite. We need to show that $\gamma \in \mathscr Z(\Gamma)$. There are two cases to consider.

Firstly, assume that $d = 1$. Simply write $\mathbf G = \mathbf G_1$. By assumption on $\mathbf G$, since $\mathscr Z_\Gamma(\gamma) < \Gamma$ has finite index, $\mathscr Z_\Gamma(\gamma) < G$ is a lattice and so $\mathscr Z_\Gamma(\gamma)$ is Zariski dense in $\mathbf G$ by Borel's density theorem \cite{Bo60} (see also \cite{Sh97, BDL14}). Since $\mathscr Z_\Gamma(\gamma) < \mathscr Z_{\mathbf G}(\gamma)$, this implies that $\mathbf G = \mathscr Z_{\mathbf G}(\gamma)$ and so $\gamma \in \mathscr Z(\mathbf G)$, hence $\gamma \in \mathscr Z(\Gamma)$.

Secondly, assume that $d \geq 2$. Let $i \in \{1, \dots, d\}$, set $G_i = \mathbf G_i(k_i)$ and denote by $q_i : G \to G_i$ the quotient homomorphism. By assumption, $q_i(\Gamma) < G_i$ is dense for the analytic topology. Since $\mathscr Z_\Gamma(\gamma) < \Gamma$ has finite index, it follows that the closure of $q_i(\mathscr Z_\Gamma(\gamma))$ in $G_i$ for the analytic topology has finite index in $G_i$. Since $G_i$ has no proper finite index subgroup (see \cite[Corollary I.1.5.7]{Ma91}), $q_i(\mathscr Z_\Gamma(\gamma)) < G_i$ is dense for the analytic topology. Since $q_i(\mathscr Z_\Gamma(\gamma)) < \mathscr Z_{G_i}(q_i(\gamma))$, this implies that $G_i =  \mathscr Z_{G_i}(q_i(\gamma))$ and so $q_i(\gamma) \in \mathscr Z(G_i)$. Since this holds for every $i \in \{1, \dots, d\}$, this further implies that $\gamma \in \mathscr Z(G)$ and so $\gamma \in \mathscr Z(\Gamma)$.
\end{proof}

By assumption, the locally compact group $G$ is unimodular. The closed subgroup $H < G$ is also unimodular because $S$ is central in $H$ and the quotient group $H/S$ is compact. Thus, \cite[Corollary B.1.7]{BHV08} implies that the homogeneous space $G/H$ carries a $G$-invariant $\sigma$-finite infinite measure $m$ whose measure class coincides with the unique $G$-invariant measure class on $G/H$. We denote by $\rE : \mathscr M \to \mathscr A : \sum_{\gamma \in \Lambda} a_\gamma u_\gamma \mapsto a_e$ the canonical faithful normal conditional expectation. We may then fix a faithful normal semifinite trace $\Tr$ on $\mathscr M$ so that $\Tr \circ \rE = \Tr$ and $\Tr|_{\mathscr A} = \int_{G/H} \cdot \, {\rm d}m$.

\begin{cor}\label{cor-commutant}
The inclusion $M = \rL(\Lambda) \subset \rL(\Lambda \curvearrowright G/H) = \mathscr M$ is ergodic in the sense that $M^\prime \cap \mathscr M = \C 1$.
\end{cor}

\begin{proof}
Since $G = \prod_{i = 1}^d \mathbf G_i(k_i)$ and $H = \prod_{i = 1}^d \mathbf H_i(k_i)$, where $\mathbf G_1, \dots, \mathbf G_d$ are chosen as in the introduction, and since $\Gamma < G$ is a lattice, \cite[Theorem 6.6 and Corollary 6.7]{BG14} imply that the nonsingular action $\Gamma \curvearrowright G/H$ is metrically ergodic and so is $\Lambda \curvearrowright G/H$. By combining Proposition \ref{prop-commutant} and Lemma \ref{lem-icc}, we have $M^\prime \cap \mathscr M = \rL(\Lambda)^\prime \cap \rL(\Lambda \curvearrowright G/H) = \mathscr Z(\rL(\Lambda)) = \C 1$.
\end{proof}

\section{Proofs of the main results} 

Let $\mathscr N$ be a von Neumann algebra and denote by $\mathscr V(\mathscr N)$ the set of partial isometries of $\mathscr N$. Endow the set $\mathscr V(\mathscr N)$ with the following order relation: for all $u_1, u_2 \in \mathscr V(\mathscr N)$, we have $u_1 \leq u_2$ if and only if $u_1 = u_2 u_1^*u_1$. It is well known that the partially ordered set $(\mathscr V(\mathscr N), \leq)$ is inductive.

\subsection{Proof of Theorem \ref{main-theorem-weyl}}

We still denote by $\sigma : \Lambda \curvearrowright \mathscr M$ the conjugation action that naturally extends the action $\Lambda \curvearrowright \mathscr A$.

Let $\Theta \in \Aut_M(\mathscr M)$. Then $\Phi = \rE \circ \Theta|_{\mathscr A} : \mathscr A \to \mathscr A$ is a $\Lambda$-equivariant normal ucp map. By Theorem \ref{thm-rigidity-ucp}, there exists a unique tuple $(\alpha_w)_{w \in \mathscr W_G} \in [0, 1]^{\mathscr W_G}$ such that $\sum_{w \in \mathscr W_G} \alpha_w = 1$ and $\Phi = \sum_{w \in \mathscr W_G} \alpha_w \theta_w$. Let $w_0 \in \mathscr W_G$ be such that $\alpha_{w_0} > 0$. Upon replacing $\Theta$ by $\Theta \circ \Theta_{w_0}^{-1}$, without loss of generality, we may further assume that $\alpha_e > 0$. We need to show that $\Theta = \id_{\mathscr M}$.

Firstly, we use a convexity argument to prove that there exists a partial isometry $v \in \mathscr V(\mathscr M)$ such that $\Theta (a) v = v a $ for every $a \in \mathscr A$. Indeed, let $p \in \mathscr A$ be a nonzero finite trace projection and denote by $\mathscr C$ the weak closure of the convex hull of the set
$$\mathscr S = \left \{ \left(\Theta(u) p u^*, (\Theta_w(u) p u^*)_{w \in \mathscr W_G\setminus \{e\}}\right) \mid u \in \mathscr U(\mathscr A)\right \} \subset \mathscr M^{\oplus n},$$ 
where $n = |\mathscr W_G|$. Denote by $\tau = \Tr^{\oplus n}$ the canonical faithful normal semifinite trace on $\mathscr M^{\oplus n}$. Since $\mathscr S$ is both uniformly bounded and $\|\cdot\|_{2, \tau}$-bounded, it follows that the image of $\mathscr C$ in the Hilbert space $\rL^2(\mathscr M^{\oplus n}, \tau)$ is closed. Denote by $(c, (c_w)_{w \in \mathscr W_G\setminus \{e\}})$ the unique element of minimal $\|\cdot\|_{2, \tau}$-norm. Since for every $u \in \mathscr U(\mathscr A)$, we have
$$\left\| \left(\Theta(u)cu^*, (\Theta_w(u)c_wu^*)_{w \in \mathscr W_G\setminus \{e\}} \right) \right\|_{2, \tau} = \left\| \left(c, (c_w)_{w \in \mathscr W_G\setminus \{e\}} \right)\right\|_{2, \tau}$$
it follows that 
$$\left(\Theta(u)cu^*, (\Theta_w(u)c_wu^*)_{w \in \mathscr W_G\setminus \{e\}} \right) = \left(c, (c_w)_{w \in \mathscr W_G \setminus \{e\}} \right).$$
Let $w \in \mathscr W_G\setminus \{e\}$. Then for every $u \in \mathscr U(\mathscr A)$, we have $\theta_w(u)c_w = \Theta_w(u)c_w = c_w u$. Since the action $\mathscr W_G \curvearrowright G/H$ is essentially free (see Lemma \ref{lem-free}), it follows that $c_w = 0$. For every $u \in \mathscr U(\mathscr A)$, we have 
$$\rE(\Theta(u)p u^*) = \rE(\Theta(u)) p u^* = \alpha_e p + \sum_{w \in \mathscr W_G\setminus \{e\}} \alpha_w\Theta_w(u) p u^*.$$
Using convexity, taking limits, and the fact that $c_w = 0$ for every $w \in \mathscr W_G \setminus \{e\}$, we obtain $\rE(c) = \alpha_e p \neq 0$ and so $c \neq 0$. Write $c = v |c|$ for the polar decomposition of $c \in \mathscr M$. Then $v \neq 0$ and $\Theta(u)v = v u$ for every $u \in \mathscr U(\mathscr A)$. Therefore, we have $\Theta(a)v = va$ for every $a \in \mathscr A$. Since $\mathscr A' \cap \mathscr M = \mathscr A$, it follows that $v^*v \in \mathscr A' \cap \mathscr M = \mathscr A$ and $vv^* \in \Theta(\mathscr A)' \cap \mathscr M =\Theta(\mathscr A)$. Moreover, we have $\Theta(v^*v) v = v$ which implies that $\Theta(v^*v)vv^* = vv^*$ and so $vv^* \leq \Theta(v^*v)$. Let $r \in \mathscr A$ be the unique projection such that $vv^* = \Theta(r)$. Then we have $\Theta(r) v = v r$ which implies that $v^*v = v^*v r$ and so $v^*v \leq r$. This shows that $\Theta(v^*v) \leq \Theta(r) = vv^*$ and so $\Theta(v^*v) = vv^*$. Upon cutting down $v$ on the right hand side by a nonzero finite trace projection in $\mathscr A$, we may further assume that $v \neq 0$ and $\Tr(v^*v) = \Tr(vv^*) < +\infty$. Since $\Theta|_{\rL(\Lambda)} = \id_{\rL(\Lambda)}$, we have that $\Theta(a)\sigma_\gamma(v) = \sigma_\gamma(v) a$ for every $a \in \mathscr A$ and every $\gamma \in \Lambda$. 

Secondly, we glue together the partial isometries $\sigma_\gamma(v)$ for $\gamma \in \Lambda$ in order to construct a unitary $u \in \mathscr U(\mathscr M)$ such that $\Theta(a) = u a u^*$ for every $a \in \mathscr A$. Indeed, fix an enumeration $\Lambda = \left\{ \gamma_n \mid n \in \N \right\}$ such that $\gamma_0 = e$. For every $n \in \N$, set $p_n = \bigvee_{0 \leq k \leq n} \sigma_{\gamma_k}(v^*v) \in \mathscr A$, $q_n = \bigvee_{0 \leq k \leq n} \sigma_{\gamma_k}(vv^*) \in \Theta(\mathscr A)$ and observe that $\max(\Tr(p_n), \Tr(q_n)) < +\infty$. In particular, we have $p_n \neq 1$ and $q_n \neq 1$. By induction over $n \in \N$, we construct a nondecreasing sequence of partial isometries $(u_n)_n$ in $\mathscr V(\mathscr M)$ such that $\Theta(a)u_n = u_n a$ for every $a \in \mathscr A$, $u_n^* u_n = p_n$, $u_n u_n^* = q_n$. Set $u_0 = v$. Assume that $u_n \in \mathscr V(\mathscr M)$ has been constructed. The same reasoning as in the previous paragraph shows that $\Theta(p_n) = \Theta(u_n^*u_n) = u_nu_n^* = q_n$. Set $r_n = p_{n+1} - p_n = \sigma_{\gamma_{n + 1}}(v^*v) - \sigma_{\gamma_{n + 1}}(v^*v)p_n \in \mathscr A$ and $s_n = q_{n+1} - q_n = \sigma_{\gamma_{n + 1}}(vv^*) - \sigma_{\gamma_{n + 1}}(vv^*)q_n \in \Theta(\mathscr A)$. We have 
\begin{align*}
\Theta(r_n) &= \sigma_{\gamma_{n + 1}}(\Theta(v^*v)) - \sigma_{\gamma_{n + 1}}(\Theta(v^*v)) \Theta(p_n) \\
&= \sigma_{\gamma_{n + 1}}(vv^*) - \sigma_{\gamma_{n + 1}}(vv^*)q_n \\
&= s_n.
\end{align*}
We define the partial isometry $u_{n + 1} \in \mathscr V(\mathscr M)$ by $u_{n + 1} = u_n + s_n \sigma_{\gamma_{n + 1}}(v) r_n$ so that
$$q_nu_{n + 1}p_n = u_n \quad \text{and} \quad s_n u_{n + 1} r_n = s_n \sigma_{\gamma_{n + 1}}(v) r_n.$$
We have
\begin{align*}
(s_n \sigma_{\gamma_{n + 1}}(v) r_n)^* (s_n \sigma_{\gamma_{n + 1}}(v) r_n) &= r_n  \sigma_{\gamma_{n + 1}}(v)^* s_n  \sigma_{\gamma_{n + 1}}(v) r_n \\
&= r_n  \sigma_{\gamma_{n + 1}}(v)^* \Theta(r_n)  \sigma_{\gamma_{n + 1}}(v) r_n \\
&= r_n  \sigma_{\gamma_{n + 1}}(v^*v) r_n \\
&= r_n.
\end{align*} 
Likewise, we have 
$$(s_n \sigma_{\gamma_{n + 1}}(v) r_n) (s_n \sigma_{\gamma_{n + 1}}(v) r_n)^* = s_n.$$ 
Then it is plain to see that $\Theta(a)u_{n + 1} = u_{n + 1} a$ for every $a \in \mathscr A$, $u_{n + 1}^* u_{n + 1} = p_n + r_n = p_{n + 1}$ and $u_{n+1} u_{n+1}^* = q_n + s_n = q_{n + 1}$. By construction, we also have $u_n \leq u_{n+1}$. Therefore, by induction, we have constructed the desired nondecreasing sequence $(u_n)_n$ in $\mathscr V(\mathscr M)$. Then we may define $u = \sup u_n \in \mathscr V(\mathscr M)$ and observe that $u^*u = \bigvee_{\gamma \in \Lambda} \sigma_\gamma(v^*v)$ and $uu^* = \bigvee_{\gamma \in \Lambda} \sigma_\gamma(vv^*)$. Since the nonsingular action $\Lambda \curvearrowright G/H$ is ergodic, it follows that $u^*u = 1 = uu^*$ and so $u \in \mathscr U(\mathscr M)$. Then we have $\Theta(a) = u a u^*$ for every $a \in \mathscr A$.

Thirdly, we exploit the fact that the nonsingular action $\Lambda \times \mathscr W_G \curvearrowright G/H$ is essentially free to show that $u \in \mathscr U(\mathscr A)$ and $\Theta = \id_{\mathscr M}$. Indeed, write $u = \sum_{\gamma \in \Lambda} a_\gamma u_\gamma$ for the Fourier expansion of $u \in \mathscr A \rtimes \Lambda = \mathscr M$. Then for every $a \in \mathscr A$, we have
$$\sum_{w \in \mathscr W_G} \alpha_w \theta_w(a) = \rE(\Theta(a)) = \rE(u a u^*) = \sum_{\gamma \in \Lambda} |a_\gamma|^2 \sigma_\gamma(a).$$
Applying Proposition \ref{prop-ucp}, denote by $\beta : G/H \to \Prob(G/H) : b \mapsto \beta_b$ the essentially unique measurable map corresponding to the normal ucp map $\Phi = \rE\circ \Theta|_{\mathscr A} : \mathscr A \to \mathscr A$. Then for $\nu$-almost every $b \in G/H$, we have
$$\sum_{\gamma \in \Lambda} |a_\gamma|^2(b) \, \delta_{\gamma^{-1} b} = \beta_b = \sum_{w \in \mathscr W_G} \alpha_w \, \delta_{w^{-1} b}.$$
For $\nu$-almost every $b \in G/H$, we consider the set of atoms of the probability measure $\beta_b$. Then Lemma \ref{lem-free} further implies that for every $w \neq e$, we have $\alpha_w = 0$ and for every $\gamma \neq e$, we have $a_\gamma = 0$ $\nu$-almost everywhere. Moreover, we have $\alpha_e = 1 = |a_e|^2$ and so $u = a_e \in \mathscr U(\mathscr A)$. Thus, $\Theta|_{\mathscr A} = \id_{\mathscr A}$ and so $\Theta = \id_{\mathscr M}$. This finishes the proof of Theorem \ref{main-theorem-weyl}.

\subsection{Proof of Corollary \ref{main-cor-connes}}

Let $\Psi : \rL(\Lambda_1 \curvearrowright G_1/H_1) \to \rL(\Lambda_2 \curvearrowright G_2/H_2)$ be a surjective unital normal $\ast$-isomorphism such that $\Psi(\rL(\Lambda_1)) = \rL(\Lambda_2)$. For every $j \in \{1, 2\}$, set $M_j = \rL(\Lambda_j) \subset \rL(\Lambda_j \curvearrowright G_j/H_j) = \mathscr M_j$. The map 
$$\Aut_{M_1}(\mathscr M_1) \to \Aut_{M_2}(\mathscr M_2) : \Theta \mapsto \Psi \circ \Theta \circ \Psi^{-1}$$
is a group isomorphism. By Theorem \ref{main-theorem-weyl}, the mapping $\mathscr W_{G_1} \to \mathscr W_{G_2} : w \mapsto \rho_2^{-1}(\Psi \circ \rho_1(w) \circ \Psi^{-1})$ is a group isomorphism.

\subsection{Further rigidity results}

In this subsection, we record a von Neumann algebraic rigidity result regarding isomorphisms between group measure space von Neumann algebras arising from actions of irreducible lattices on algebraic homogeneous spaces. 

Keep the same notation as before. For every $j \in \{1, 2\}$ and every $i \in \{1, \dots, d\}$, let $\mathscr Z(\mathbf G_i) < \mathbf L_{i, j} <\mathbf G_i$ be a proper Zariski connected $k_i$-subgroup. Set $L_j = \prod_{i = 1}^d \mathbf L_{i, j}(k_i) < G$. Since $\mathscr Z(\Gamma) < L_j$, we may consider the nonsingular action $\Lambda \curvearrowright G/L_j$. Assume moreover that $L_j < G$ is noncompact. Since $\Gamma < G$ is irreducible, Moore's ergodicity theorem implies that $\Lambda \curvearrowright G/L_j$ is ergodic (see \cite{HM77}).

\begin{thm}\label{thm-rigidity-homogeneous}
Let $\Psi : \rL(\Lambda \curvearrowright G/L_1) \to \rL(\Lambda \curvearrowright G/L_2)$ be a surjective unital normal $\ast$-isomorphism such that $\Psi|_{\rL(\Lambda)} = \id_{\rL(\Lambda)}$. Then there exists $g \in G$ such that $g L_1 g^{-1} = L_2$.
\end{thm}

\begin{proof}
For every $j \in \{1, 2\}$, we may choose a Borel probability measure $\nu_j \in \Prob(G/L_j)$ whose measure class coincides with the unique $G$-invariant measure class on $G/L_j$. Denote by $\rE_j : \rL(\Lambda \curvearrowright G/L_j) \to \rL^\infty(G/L_j)$ the canonical $\Lambda$-equivariant faithful normal conditional expectation. Consider the $\Lambda$-equivariant normal ucp map $\Phi = \rE_2 \circ \Psi|_{\rL^\infty(G/L_1)} : \rL^\infty(G/L_1) \to \rL^\infty(G/L_2)$. By Proposition \ref{prop-ucp}, there exists an essentially unique $\Lambda$-equivariant measurable map $\beta : G/L_2 \to \Prob(G/L_1) : b \mapsto \beta_b$ such that $\bary(\beta_\ast \nu_2) = \nu_2 \circ \Phi \prec \nu_1$ and for $\nu_2$-almost every $b \in G/L_2$ and every $f \in \rC_0(G/L_1)$, we have $\beta_b(f) = \Phi(f)(b)$. By Theorem \ref{thm-equivariant}, we have $(G/L_1)^{L_2} \neq \emptyset$ and so there exists $g \in G$ such that $g^{-1} L_2 g < L_1$. Likewise, by considering the $\Lambda$-equivariant normal ucp map $\rE_1 \circ \Psi^{-1}|_{\rL^\infty(G/L_2)} : \rL^\infty(G/L_2) \to \rL^\infty(G/L_1)$, there exists $h \in G$ such that $h^{-1} L_1 h < L_2$. Set $x = gh \in G$ and observe that $x^{-1} L_2 x < L_2$. Write $x = (x_i)_i \in \prod_{i = 1}^d \mathbf G_i(k_i)$. For every $i \in \{1, \dots, d\}$, since $\mathbf L_{i, 2}(k_i)$ is Zariski dense in $\mathbf L_{i, 2}$ by \cite[Proposition I.2.5.3]{Ma91}, we have $x_i^{-1} \mathbf L_{i, 2} x_i < \mathbf L_{i, 2}$. By the descending chain condition, we have $x_i^{-1} \mathbf L_{i, 2} x_i = \mathbf L_{i, 2}$ and so $x_i^{-1} \mathbf L_{i, 2}(k_i) x_i = \mathbf L_{i, 2}(k_i)$. This further implies that $x^{-1} L_2 x = L_2$ and so $g^{-1} L_2 g = L_1$.
\end{proof}

Let us point out that the assumption that $\Psi|_{\rL(\Lambda)} = \id_{\rL(\Lambda)}$ cannot be dropped in general. Indeed, when $L_1$ and $L_2$ are amenable and unimodular, the group measure space von Neumann algebras $\rL(\Lambda \curvearrowright G/L_1)$ and $\rL(\Lambda \curvearrowright G/L_2)$ are both $\ast$-isomorphic to the unique AFD type ${\rm II_\infty}$ factor.

\begin{rem}
We point out that Theorem \ref{main-theorem-weyl} holds more generally when for every $i \in \{1, \dots, d\}$, $\mathbf H_i = \mathscr Z_{\mathbf G_i}(\mathbf S_i) < \mathbf G_i$ is replaced by a Zariski connected $k_i$-subgroup $\mathbf L_i < \mathbf G_i$ for which the quotient $\mathscr N_{\mathbf G_i}(\mathbf L_i)/\mathbf L_i$ is finite, every coset of $\mathscr N_{\mathbf G_i}(\mathbf L_i)/\mathbf L_i$ is represented by an element rational over $k_i$ and $L = \prod_{i = 1}^{d}\mathbf L_i(k_i)$ is noncompact and unimodular. 
\end{rem}

\bibliographystyle{plain}

\end{document}